\documentclass[11pt,reqno]{amsart}
\usepackage{amsmath,amssymb}

 \makeatletter
 \evensidemargin  \oddsidemargin
 \makeatother

\begin{document}
\thispagestyle{empty}
 \setcounter{page}{1}

\begin{center}
{\large\bf Stability of a  functional equation deriving from cubic
and quartic  functions

\vskip.20in

{\bf  M. Eshaghi Gordji } \\[2mm]

{\footnotesize Department of Mathematics,
Semnan University,\\ P. O. Box 35195-363, Semnan, Iran\\
[-1mm] e-mail: {\tt madjid.eshaghi@gmail.com}}

{\bf  A. Ebadian } \\[2mm]
 {\footnotesize Department of Mathematics,
Urmia  University, Urmia, Iran\\
[-1mm] e-mail: {\tt ebadian.ali@gmail.com}}

{\bf  S. Zolfaghari } \\[2mm]

{\footnotesize Department of Mathematics,
Semnan University,\\ P. O. Box 35195-363, Semnan, Iran\\
[-1mm] e-mail: {\tt zolfagharys@Yahoo.com}}}
\end{center}
\vskip 5mm
 \noindent{\footnotesize{\bf Abstract.} In this paper, we obtain the general solution and the generalized
 Ulam-Hyers  stability of the cubic and quartic functional equation
\begin{align*}
&4(f(3x+y)+f(3x-y))=-12(f(x+y)+f(x-y))\\&+12(f(2x+y)+f(2x-y))-8f(y)-192f(x)+f(2y)+30f(2x).
\end{align*}

\vskip.10in

  \newtheorem{df}{Definition}[section]
  \newtheorem{rk}[df]{Remark}
   \newtheorem{lem}[df]{Lemma}
   \newtheorem{thm}[df]{Theorem}
   \newtheorem{pro}[df]{Proposition}
   \newtheorem{cor}[df]{Corollary}
   \newtheorem{ex}[df]{Example}

 \setcounter{section}{0}
 \numberwithin{equation}{section}

\vskip .2in

\begin{center}
\section{Introduction}
\end{center}
The stability problem of functional equations originated from a
question of Ulam [34] in 1940, concerning the stability of group
homomorphisms. Let $(G_1,.)$ be a group and let $(G_2,*)$ be a
metric group with the metric $d(.,.).$ Given $\epsilon >0$, does
there exist a $\delta
>0$, such that if a mapping $h:G_1\longrightarrow G_2$ satisfies the
inequality $d(h(x.y),h(x)*h(y)) <\delta$ for all $x,y\in G_1$,
then there exists a homomorphism $H:G_1\longrightarrow G_2$ with
$d(h(x),H(x))<\epsilon$ for all $x\in G_1?$ In the other words,
Under what condition does there exists a homomorphism near an
approximate homomorphism? The concept of stability for functional
equation arises when we replace the functional equation by an
inequality which acts as a perturbation of the equation. In 1941,
D. H. Hyers [9] gave the first affirmative  answer to the question
of Ulam for Banach spaces. Let $f:{E}\longrightarrow{E'}$ be a
mapping between Banach spaces such that
$$\|f(x+y)-f(x)-f(y)\|\leq \delta $$
for all $x,y\in E,$ and for some $\delta>0.$ Then there exists a
unique additive mapping $T:{E}\longrightarrow{E'}$ such that
$$\|f(x)-T(x)\|\leq \delta$$
for all $x\in E.$ Moreover if $f(tx)$ is continuous in
$t\in\mathbb{R}$  for each fixed $x\in E,$ then $T$ is linear.
Finally in 1978, Th. M. Rassias [31] proved the following theorem.

\begin{thm}\label{t1} Let $f:{E}\longrightarrow{E'}$ be a mapping from
 a normed vector space ${E}$
into a Banach space ${E'}$ subject to the inequality
$$\|f(x+y)-f(x)-f(y)\|\leq \epsilon (\|x\|^p+\|y\|^p) \eqno \hspace {0.5
 cm} (1.1)$$
for all $x,y\in E,$ where $\epsilon$ and p are constants with
$\epsilon>0$ and $p<1.$ Then there exists a unique additive mapping
$T:{E}\longrightarrow{E'}$ such that
$$\|f(x)-T(x)\|\leq \frac{2\epsilon}{2-2^p}\|x\|^p  \eqno \hspace {0.5
 cm}(1.2)$$ for all $x\in E.$
If $p<0$ then inequality (1.1) holds for all $x,y\neq 0$, and
(1.2) for $x\neq 0.$ Also, if the function $t\mapsto f(tx)$ from
$\Bbb R$ into $E'$ is continuous in real $t$ for each fixed $x\in
E,$ then T is linear.
\end{thm}
In 1991, Z. Gajda [5] answered the question for the case $p>1$,
which was raised by Rassias.  This new concept is known as
Hyers-Ulam-Rassias stability of functional equations (see [1-2],
[5-10], [28-30]). On the other hand J. M. Rassias [25-27],
generalized the Hyers stability result by presenting a weaker
condition controlled by a product of different powers of norms.
According to J. M. Rassias Theorem:

\begin{thm}\label{t2} If it is assumed that there exist constants $\Theta\geq 0$ and
$p_1,p_2\in \Bbb R$ such that $p=p_1+p_2\neq 1,$ and
$f:{E}\longrightarrow{E'}$ is a map from a norm space ${E}$ into a
Banach space ${E'}$ such that the inequality
$$\|f(x+y)-f(x)-f(y)\|\leq \epsilon \|x\|^{p_1}\|y\|^{p_2} \eqno (1.1p)$$
for all $x,y\in E,$ then there exists a unique additive mapping
$T:{E}\longrightarrow{E'}$ such that
$$\|f(x)-T(x)\|\leq \frac{\Theta}{2-2^p}\|x\|^p, $$ for all $x\in E.$
If in addition for every $x\in E$, $f(tx)$  is continuous in real t
for each fixed x, then $T$ is linear (see [19-26]).
\end{thm}

The oldest cubic functional equation, and was introduced by J. M.
Rassias (in 2000-2001): [17-18], as follows:
$$f(x+2y)+3f(x)=3f(x+y)+f(x-y)+6f(y).$$
 Jun and Kim [11] introduced the following cubic
functional equation
 $$f(2x+y)+f(2x-y)=2f(x+y)+2f(x-y)+12f(x) \eqno(1.3)$$
and they established the general solution and the generalized
Hyers-Ulam-Rassias stability for the  functional equation (1.3).
The function $f(x)=x^3$ satisfies the functional equation (1.3),
which is thus called a cubic functional equation. Every solution
of the cubic functional equation is said to be a cubic function.
Jun and Kim proved that  a function
 $f$ between real vector spaces X and Y is a solution of (1.3) if and only if there
 exists a unique function $C:X\times X\times X\longrightarrow Y$ such that
 $f(x)=C(x,x,x)$ for all $x\in X,$ and $C$ is symmetric for each
 fixed one variable and is additive for fixed two variables.
The oldest quartic functional equation, and was introduced by J. M.
Rassias(in 1999-2000): [16], [27], and then (in 2005) was employed
by Won-Gil Park [15] and others, such that:

$$f(x+2y)+f(x-2y)=4(f(x+y)+f(x-y))+24f(y)-6f(x).\eqno \hspace {0.5cm}(1.4)$$
In fact they proved that a function
 $f$ between real vector spaces X and Y is a solution of (1.4) if and only if there
 exists a unique symmetric multi-additive function $Q:X\times X\times X\times X\longrightarrow Y$ such that
 $f(x)=Q(x,x,x,x)$ for all $x$ (see also [3,4], [12-15], [33]). It is easy to show that
 the function $f(x)=x^4$ satisfies the functional equation (1.4), which is called a
 quartic functional equation and every solution of the quartic functional equation is said to be a
 quartic function.

We deal with the following functional equation deriving from
quartic and
 cubic functions:
\begin{align*}
4(f(3x+y)+f(3x-y))&=-12(f(x+y)+f(x-y))+12(f(2x+y)+f(2x-y))\\
&-8f(y)-192f(x)+f(2y)+30f(2x). \hspace {3cm}(1.5)
\end{align*}
It is easy to see that
 the function $f(x)=ax^4+bx^3$ is a solution of the functional equation (1.5). In the
 present paper we investigate the general solution and the generalized
  Hyers-Ulam-Rassias stability of the functional equation (1.5).

\vskip 5mm \begin{center}
\section{ General solution}
\end{center}
In this section we establish the general solution of functional
equation (1.5).

\begin{thm}\label{t2}
Let $X$,$Y$ be vector spaces,  and let  $f:X\longrightarrow Y$  be
a function. Then $f$ satisfies (1.5) if and only if there exists a
unique symmetric multi-additive  function $Q:X\times X \times
X\times X \longrightarrow Y$ and a unique  function $C:X\times
X\times X \rightarrow Y$
 such that
$C$ is symmetric for each
 fixed one variable and is additive for fixed two variables, and
 that  $f(x)=Q(x,x,x,x)+C(x,x,x)$ for all $x\in X.$
\end{thm}
\begin{proof} Suppose there exists a
 symmetric multi-additive  function $Q:X\times X \times
X\times X \longrightarrow Y$ and a  function $C:X\times X\times X
\rightarrow Y$
 such that
$C$ is symmetric for each
 fixed one variable and is additive for fixed two variables, and
 that
 $f(x)=Q(x,x,x,x)+C(x,x,x)$ for all $x\in X.$ Then it is easy to see
 that $f$ satisfies (1.5). For the converse let  $f$ satisfy (1.5). We decompose $f$ into the even part and odd
part by setting

$$f_e(x)=\frac{1}{2}(f(x)+f(-x)),~~\hspace {0.3 cm}f_o(x)=\frac{1}{2}(f(x)-f(-x)),$$
for all $x\in X.$ By (1.5), we have
\begin{align*}
4f_e(3x+y)&+4f_e(3x-y)=\frac{1}{2}[4f(3x+y)+4f(-3x-y)+4f(3x-y)+4f(-3x+y)]\\
&=\frac{1}{2}[4f(3x+y)+4f(3x-y)]+\frac{1}{2}[4f((-3x)+(-y))+4f((-3x)-(-y))]\\
&=\frac{1}{2}[12f(2x+y)+12f(2x-y)-12f(x+y)-12f(x-y)\\
&-8f(y)-192f(x)+f(2y)+30f(2x)]\\
&+\frac{1}{2}[12f(-2x-y)+12f((-2x)+y))-12f(-x-y)-12f(-x+y)\\
&-8f(-y)-192f(-x)+f(-2y)+30f(-2x)]\\
&=12[\frac{1}{2}(f(2x+y)+f(-(2x+y)))]+12[\frac{1}{2}(f(2x-y)+f(-(2x-y)))]\\
&-12[\frac{1}{2}(f(x+y)+f(-(x+y)))]-12[\frac{1}{2}(f(x-y)+f(-(x-y)))]\\
&-8[\frac{1}{2}(f(y)+f(-y))]-192[\frac{1}{2}(f(x)+f(-x))]\\
&+\frac{1}{2}[f(2y)+f(-2y)]+30[\frac{1}{2}(f(2x)+f(-2x))]\\
&=12(f_e(2x+y)+f_e(2x-y))-12(f_e(x+y)+f_e(x-y))\\
&-8f_e(y)-192f_e(x)+f_e(2y)+30f_e(2x)
\end{align*}
for all $x,y\in X.$ This means that $f_e$ satisfies (1.5), or
\begin{align*}
4(f_e(3x+y)+f_e(3x-y))&=-12(f_e(x+y)+f_e(x-y))+12(f_e(2x+y)+f_e(2x-y))\\
&-8f_e(y)-192f_e(x)+f_e(2y)+30f_e(2x). \hspace {2.5cm}(1.5(e))
\end{align*}
Now putting $x=y=0$ in (1.5(e)), we get $f_e(0)=0$. Setting $x=0$ in
(1.5(e)), by evenness of $f_e$ we obtain

$$f_e(2y)=16f_e(y) \eqno\hspace {4.5cm}(2.1)$$
for all $y\in X.$ Hence (1.5(e)) can be written as

\begin{align*}
&f_e(3x+y)+f_e(3x-y)+3(f_e(x+y)\\&+f_e(x-y))=3(f_e(2x+y)+f_e(2x-y))+72f_e(x)+2f_e(y)
\hspace {3.6cm}(2.2)
\end{align*}
for all $x,y \in X.$ With the substitution $y:=2y$ in (2.2), we have

\begin{align*}
&f_e(3x+2y)+f_e(3x-2y)+3f_e(x+2y)\\&+3f_e(x-2y)=48f_e(x+y)+48f_e(x-y)+72f_e(x)+32f_e(y).
\hspace {3.3cm} (2.3)
\end{align*}
Replacing $y$ by $x+2y$ in (2.2), we obtain

\begin{align*}
&16f_e(2x+y)+16f_e(x-y)+48f_e(x+y)\\&+48f_e(y)=3f_e(3x+2y)+3f_e(x-2y)+2f_e(x+2y)+72f_e(x)
.\hspace {3.3cm} (2.4)
\end{align*}
Substituting  $-y$ for $y$ in (2.4) gives

\begin{align*}
&16f_e(2x-y)+16f_e(x+y)+48f_e(x-y)\\&+48f_e(y)=3f_e(3x-2y)+3f_e(x+2y)+2f_e(x-2y)+72f_e(x)
.\hspace {3.3cm} (2.5)
\end{align*}
By utilizing equations (2.3), (2.4) and (2.5), we obtain

\begin{align*}
&4f_e(2x+y)+4f_e(2x-y)+f_e(x+2y)\\&+f_e(x-2y)=20f_e(x+y)+20f_e(x-y)+90f_e(x)
.\hspace {5.1cm} (2.6)
\end{align*}
Interchanging $x$ and $y$ in (2.3), we get

\begin{align*}
&f_e(2x+3y)+f_e(2x-3y)+3f_e(2x+y)\\&+3f_e(2x-y)=48f_e(x+y)+48f_e(x-y)+32f_e(x)+72f_e(y)
.\hspace {3.5cm} (2.7)
\end{align*}
If we add (2.3) to (2.7), we have

\begin{align*}
&f_e(2x+3y)+f_e(3x+2y)+f_e(2x-3y)+f_e(3x-2y)\\
&+3f_e(2x+y)+3f_e(x+2y)+3f_e(2x-y)\\
&+3f_e(x-2y)=96f_e(x+y)+96f_e(x-y)+104f_e(x)+104f_e(y). \hspace
{3.1cm} (2.8)
\end{align*}
And by utilizing equations (2.4), (2.5) and (2.8), we arrive at

\begin{align*}
&3f_e(2x+3y)+3f_e(2x-3y)=-25f_e(2x+y)-25f_e(2x-y)\\
&-4f_e(x-2y)-4f_e(x+2y)+224f_e(x+y)\\
&+224f_e(x-y)+456f_e(x)+216f_e(y). \hspace {7cm} (2.9)
\end{align*}
Let us interchange $x$ and $y$ in (2.9). Then we see that

\begin{align*}
&3f_e(3x+2y)+3f_e(3x-2y)=-25f_e(x+2y)-25f_e(x-2y)\\
&-4f_e(2x-y)-4f_e(2x+y)+224f_e(x+y)\\
&+224f_e(x-y)+456f_e(y)+216f_e(x).\hspace {7cm} (2.10)
\end{align*}
Comparing (2.10) with (2.3), we get

\begin{align*}
&4f_e(2x-y)+4f_e(2x+y)=-16f_e(x+2y)\\&-16f_e(x-2y)+80f_e(x+y)+80f_e(x-y)+360f_e(y)
.\hspace {4.6cm} (2.11)
\end{align*}
If we compare (2.11) and (2.6), we conclude that

$$f_e(x+2y)+f_e(x-2y)+6f_e(x)=4f_e(x+y)+4f_e(x-y)+24f_e(y).$$ This means that
$f_e$ is quartic  function. Thus there exists  a unique symmetric
multi-additive function $Q:X\times X\times X\times X\longrightarrow
Y$ such that $f_e(x)=Q(x,x,x,x)$ for all $x\in X.$ On the other hand
we can show that $f_o$ satisfies (1.5), or
\begin{align*}
4(f_o(3x+y)+f_o(3x-y))&=-12(f_o(x+y)+f_o(x-y))+12(f_o(2x+y)+f_o(2x-y))\\
&-8f_o(y)-192f_o(x)+f_o(2y)+30f_o(2x). \hspace {2.4cm}(1.5(o))
\end{align*}
 Now setting $x=y=0$ in (1.5(o)) gives
 $f_o(0)=0.$ Putting $x=0$ in (1.5(o)), then
 by oddness of $f_o$,  we have

$$f_o(2y)=8f_o(y). \eqno\hspace {3cm}(2.12)$$ Hence (1.5(o)) can be written as

\begin{align*}
&f_o(3x+y)+f_o(3x-y)+3f_o(x+y)\\&+3f_o(x-y)=3f_o(2x+y)+3f_o(2x-y)+12f_o(x)
\hspace {4.6cm} (2.13)
\end{align*}
for all $x,y\in X.$ Replacing $x$ by $x+y$, and $y$ by $x-y$ in
(2.13) we have

\begin{align*}
&8f_o(2x+y)+8f_o(x+2y)+24f_o(x)\\&+24f_o(y)=3f_o(3x+y)+3f_o(x+3y)+12f_o(x+y)
\hspace {4.6cm}(2.14)
\end{align*}
and interchanging $x$ and $y$ in (2.13) yields

\begin{align*}
&f_o(x+3y)-f_o(x-3y)+3f_o(x+y)\\&-3f_o(x-y)=3f_o(x+2y)-3f_o(x-2y)+12f_o(y).
\hspace {4.6cm}(2.15)
\end{align*}
Which on substitution of $-y$ for $y$ in (2.13) gives

\begin{align*}
&f_o(3x-y)+f_o(3x+y)+3f_o(x-y)\\&+3f_o(x+y)=3f_o(2x-y)+3f_o(2x+y)+12f_o(x)
.\hspace {4.6cm} (2.16)
\end{align*}
Replace $y$ by $x+2y$ in (2.13).  Then we have

\begin{align*}
&8f_o(2x+y)+8f_o(x-y)+24f_o(x+y)\\&-24f_o(y)=3f_o(3x+2y)+3f_o(x-2y)+12f_o(x)
.\hspace {4.6cm} (2.17)
\end{align*}
 From the substitution $y:=-y$ in (2.17) it follows that

\begin{align*}
&8f_o(2x-y)+8f_o(x+y)+24f_o(x-y)\\&+24f_o(y)=3f_o(3x-2y)+3f_o(x+2y)+12f_o(x)
.\hspace {4.6cm} (2.18)
\end{align*}
If we add (2.17) to (2.18),  we have

\begin{align*}
&3f_o(3x-2y)+3f_o(3x+2y)=8f_o(2x+y)+8f_o(2x-y)\\
&-3f_o(x+2y)-3f_o(x-2y)+32f_o(x-y)\\
&+32f_o(x+y)-24f_o(x). \hspace {8.3cm} (2.19)
\end{align*}
Let us interchange $x$ and $y$ in (2.19). Then we see that

\begin{align*}
&3f_o(2x+3y)-3f_o(2x-3y)=8f_o(x+2y)-8f_o(x-2y)\\
&-3f_o(2x+y)+3f_o(2x-y)+32f_o(x+y)\\
&-32f_o(x-y)-24f_o(y). \hspace {8.3cm} (2.20)
\end{align*}
With the substitution $y:=x+y$ in (2.13),  we have

\begin{align*}
&f_o(4x+y)+f_o(2x-y)+3f_o(2x+y)\\&-3f_o(y)=3f_o(3x+y)+3f_o(x-y)+12f_o(x)\hspace
{5.3cm} (2.21)
\end{align*}
and replacing $-y$ by $y$ gives

\begin{align*}
&f_o(4x-y)+f_o(2x+y)+3f_o(2x-y)\\&+3f_o(y)=3f_o(3x-y)+3f_o(x+y)+12f_o(x).\hspace
{5.3cm} (2.22)
\end{align*}
If we add (2.21) to (2.22), we have

\begin{align*}
&f_o(4x+y)+f_o(4x-y)=3f_o(3x+y)+3f_o(3x-y)\\
&-4f_o(2x-y)-4f_o(2x+y)+3f_o(x-y)\\
&+3f_o(x+y)+24f_o(x).\hspace {8.6cm} (2.23)
\end{align*}
By comparing (2.16) with (2.23), we arrive at

\begin{align*}
&f_o(4x+y)+f_o(4x-y)=5f_o(2x+y)+5f_o(2x-y)\\&-6f_o(x+y)-6f_o(x-y)+60f_o(x)\hspace
{6.7cm} (2.24)
\end{align*}
and replacing $y$ by $2y$ in (2.13) gives

\begin{align*}
&f_o(3x+2y)+f_o(3x-2y)=24f_o(x+y)+24f_o(x-y)\\&-3f_o(x+2y)-3f_o(x-2y)+12f_o(x).\hspace
{6.4cm} (2.25)
\end{align*}
By comparing (2.25) with (2.19), we arrive at

\begin{align*}
&3f_o(x+2y)+3f_o(x-2y)=20f_o(x+y)+20f_o(x-y)\\&-4f_o(2x+y)-4f_o(2x-y)+30f_o(x)
.\hspace {6.5cm} (2.26)
\end{align*}
Let us interchange $x$ and $y$ in (2.25).  Then we see that

\begin{align*}
&f_o(2x+3y)-f_o(2x-3y)=24f_o(x+y)-24f_o(x-y)\\&-3f_o(2x+y)+3f_o(2x-y)+12f_o(y)
.\hspace {6.5cm} (2.27)
\end{align*}
Thus combining (2.27) with (2.20) yields

\begin{align*}
&4f_o(x+2y)-4f_o(x-2y)=3f_o(2x-y)-3f_o(2x+y)\\&+20f_o(x+y)-20f_o(x-y)+30f_o(y)
.\hspace {6.5cm} (2.28)
\end{align*}
By comparing (2.28) with (2.15), we arrive at

\begin{align*}
&4f_o(x+3y)-4f_o(x-3y)=9f_o(2x-y)-9f_o(2x+y)\\&+48f_o(x+y)-48f_o(x-y)+138f_o(y)
.\hspace {6.3cm} (2.29)
\end{align*}
Which, by putting $y:=2y$ in (2.14), leads to

\begin{align*}
&64f_o(x+y)+8f_o(x+4y)+24f_o(x)\\&+192f_o(y)=3f_o(3x+2y)+3f_o(x+6y)+12f_o(x+2y)
.\hspace {3.9cm} (2.30)
\end{align*}
Replacing $y$ by $-y$ in (2.30) gives

\begin{align*}
&64f_o(x-y)+8f_o(x-4y)+24f_o(x)\\&-192f_o(y)=3f_o(3x-2y)+3f_o(x-6y)+12f_o(x-2y)
.\hspace {3.9cm} (2.31)
\end{align*}
If we subtract (2.30) from (2.31), we obtain

\begin{align*}
&8f_o(x+4y)-8f_o(x-4y)=3f_o(3x+2y)-3f_o(3x-2y)\\
&+3f_o(x+6y)-3f_o(x-6y)+12f_o(x+2y)-12f_o(x-2y)\\
&+64f_o(x-y)-64f_o(x+y)-384f_o(y). \hspace {6.3cm} (2.32)
\end{align*}
Setting $x$ instead of $y$ and $y$ instead of $x$ in (2.24), we get

\begin{align*}
&f_o(x+4y)-f_o(x-4y)=5f_o(x+2y)-5f_o(x-2y)\\&+6f_o(x-y)-6f_o(x+y)+60f_o(y).\hspace
{6.7cm} (2.33)
\end{align*}
Combining (2.32) and (2.33) yields

\begin{align*}
&3f_o(3x+2y)-3f_o(3x-2y)=28f_o(x+2y)-28f_o(x-2y)\\
&+3f_o(x-6y)-3f_o(x+6y)+16f_o(x+y)\\
&-16f_o(x-y)+864f_o(y) \hspace {8.3cm} (2.34)
\end{align*}
and  subtracting (2.18) from (2.17), we obtain

\begin{align*}
&3f_o(3x+2y)-3f_o(3x-2y)=3f_o(x+2y)-3f_o(x-2y)\\
&+8f_o(2x+y)-8f_o(2x-y)+16f_o(x+y)\\
&-16f_o(x-y)-48f_o(y). \hspace {8.3cm} (2.35)
\end{align*}
By comparing (2.34) with (2.35), we arrive at

\begin{align*}
&3f_o(x+6y)-3f_o(x-6y)=25f_o(x+2y)-25f_o(x-2y)\\&+8f_o(2x-y)-8f_o(2x+y)+912f_o(y)
.\hspace {6.3cm} (2.36)
\end{align*}
Interchanging $y$ with $2y$ in (2.29) gives the equation

\begin{align*}
&4f_o(x+6y)-4f_o(x-6y)=48f_o(x+2y)-48f_o(x-2y)\\&+72f_o(x-y)-72f_o(x+y)+1104f_o(y)
.\hspace {6.3cm} (2.37)
\end{align*}
We obtain from (2.36) and (2.37)

\begin{align*}
&44f_o(x+2y)-44f_o(x-2y)=32f_o(2x-y)-32f_o(2x+y)\\&+216f_o(x+y)-216f_o(x-y)+336f_o(y)
.\hspace {6cm} (2.38)
\end{align*}
By using (2.28) and (2.38), we lead to

$$f_o(2x+y)-f_o(2x-y)=4f_o(x+y)-4f_o(x-y)-6f_o(y). \eqno\hspace {0.5cm} (2.39)$$
And interchanging $x$ with $y$ in (2.39) gives

$$f_o(x+2y)+f_o(x-2y)=4f_o(x+y)+4f_o(x-y)-6f_o(x). \eqno\hspace {0.5cm} (2.40)$$
If we compare (2.40) and (2.26), we conclude that

$$8f_o(x+y)+8f_o(x-y)+48f_o(x)=4f_o(2x+y)+4f_o(2x-y). \hspace {0.5cm}$$
This means that $f_o$ is cubic function and that there exits a
unique function $C:X\times X\times X\longrightarrow Y$ such that
 $f_o(x)=C(x,x,x)$ for all $x\in X,$ and $C$ is symmetric for each
 fixed one variable and is additive for fixed two variables. Thus for all $x\in X$, we have
$$f(x)=f_e(x)+f_o(x)=C(x,x,x)+Q(x,x,x,x).$$  This
 completes the proof of Theorem.
\end{proof}
The following Corollary is an alternative result of above Theorem
2.1.

\begin{cor}\label{c1} Let $X$,$Y$
be vector spaces,  and let  $f:X\longrightarrow Y$  be a function
satisfying (1.5). Then the following assertions hold.

a) If f is even function, then f is quartic.

b) If f is odd
function, then f is cubic.
\end{cor}

\section{ Stability }

We now investigate the generalized Hyers-Ulam-Rassias stability
problem for functional equation (1.5).  From now on,  let X be a
real vector space and let Y be a Banach space. Now before taking up
the main subject, given $f:X\rightarrow Y$, we define the difference
operator $D_f:X\times X \rightarrow Y$ by

\begin{align*}
D_{f}(x,y)&=4[f(3x+y)+f(3x-y)]-12[f(2x+y)+f(2x-y)]+12[f(x+y)+f(x-y)]\\
&-f(2y)+8f(y)-30f(2x)+192f(x)
\end{align*}
for all $x,y \in X.$  We consider the following functional
inequality :

$$\|D_f(x,y)\|\leq\phi(x,y)$$
for an upper bound $\phi:X \times X\rightarrow [0,\infty).$
\begin{thm}\label{t2} Let $s\in \{1,-1\}$ be fixed. Suppose that an even mapping $f:X\rightarrow Y$ satisfies $f(0)=0,$ and
 $$\|D_f(x,y)\|\leq \phi(x,y) \eqno\hspace {0.5cm}(3.1)$$
 for all $x,y\in X. $ If the upper bound $\phi:X\times X\rightarrow
 [0,\infty)$ is a mapping such that the series
 $\sum^{\infty}_{i=0} 2^{4si} \phi(0,\frac{x}{2^{si}})$ converges,
  and that $\lim_{n\rightarrow\infty} 2^{4sn}
 \phi(\frac{x}{2^{sn}},\frac{y}{2^{sn}})=0$ for all $x,y\in X, $
 then the limit $Q(x)=\lim_{n\rightarrow\infty} 2^{4sn} f(\frac{x}{2^{sn}})$ exists for
 all $x\in X, $ and $Q:X\rightarrow Y$ is a unique quartic
 function satisfying (1.5), and
 $$ \|f(x)-Q(x)\|\leq \frac{1}{16}\sum^{\infty}_{i=\frac{s-1}{2}} 2^{4s(i+1)}
  \phi(0,\frac{x}{2^{s(i+1)}}) \eqno\hspace {0.5cm} (3.2)$$
 for all $x\in X. $

\end{thm}
\begin{proof}
Let $s=1.$   Putting $x=0$ in (3.1),  we get

$$\|f(2y)-16f(y)\|\leq \phi(0,y). \eqno\hspace {0.5cm}(3.3)$$
Replacing $y$ by $\frac{x}{2}$ in (3.3), yields

$$\|f(x)-16f(\frac{x}{2})\|\leq\phi(0,\frac{x}{2}). \eqno\hspace {0.5cm}(3.4)$$
Interchanging $x$ with $\frac{x}{2}$ in (3.4),  and multiplying by
16 it follows that

$$\|16f(\frac{x}{2})-16^2f(\frac{x}{4})\|\leq 16   \phi(0,\frac{x}{4}).  \eqno\hspace
{0.5cm} (3.5)$$ Combining (3.4) and (3.5), we lead to

$$\parallel 16^2f(\frac{x}{4})-f(x)\parallel\leq
\phi(0,\frac{x}{2})+ 16 \phi(0,\frac{x}{4}). \eqno\hspace
{0.5cm}(3.6)$$ From the inequality (3.4) we use iterative methods
and induction on $n$ to prove our next relation:

$$\parallel 16^nf(\frac{x}{2^n})-f(x)\parallel\leq
\frac{1}{16}\sum^{n-1}_{i=0} 16^{i+1} \phi(0,\frac{x}{2^{i+1}}).
\eqno\hspace {0.5cm} (3.7)$$ We multiply (3.7) by $16^m$ and
replace $x$ by $\frac{x}{2^m}$ to obtain that
$$\parallel 16^{m+n}f(\frac{x}{2^{m+n}})-16^m f(\frac{x}{2^m})\parallel\leq
\sum^{n-1}_{i=0} 16^{m+i} \phi(0,\frac{x}{2^{i+m+1}}).$$ This
shows that $\{16^nf(\frac{x}{2^n})\}$ is a Cauchy sequence in Y by
taking the limit $m\rightarrow\infty$.  Since Y is a Banach space,
it follows that the sequence $\{16^nf(\frac{x}{2^n})\}$ converges.
We define $Q:X\rightarrow Y$ by $Q(x)=\lim_{n\rightarrow\infty}
2^{4n} f(\frac{x}{2^n})$ for all $x\in X.$ It is clear that
$Q(-x)=Q(x)$ for all $x\in X$,  and it follows from (3.1) that
$$\parallel D_Q(x,y)\parallel=\lim_{n\rightarrow\infty} 16^n \parallel
D_f(\frac{x}{2^n},\frac{y}{2^n})\parallel\leq\lim_{n\rightarrow\infty}
16^n \phi(\frac{x}{2^n},\frac{y}{2^n})=0$$ for all $x,y\in X.$
Hence by Corollary 2.2, Q is quartic.  It remains to show that Q
is unique. Suppose that there exists another quartic function
$Q':X\rightarrow Y$ which satisfies (1.5) and (3.2). Since
$Q(2^nx)=16^n Q(x)$,  and $Q'(2^nx)=16^n Q'(x)$ for all $x\in X, $
we conclude that

\begin{align*}
\parallel Q(x)-Q'(x)\parallel
&=16^n\parallel Q(\frac{x}{2^n})-Q'(\frac{x}{2^n})\parallel\\
&\leq16^n\parallel
Q(\frac{x}{2^n})-f(\frac{x}{2^n})\parallel+16^n\parallel
Q'(\frac{x}{2^n})-f(\frac{x}{2^n})\parallel\\
&\leq 2\sum^{\infty}_{i=0} 16^{n+i} \phi(0,\frac{x}{2^{n+i+1}})
\end{align*}
for all $x\in X. $ By letting $n\rightarrow\infty$ in this
inequality,  it follows that $Q(x)=Q'(x)$ for all $x\in X, $ which
gives the conclusion.  For $s=-1$, we obtain

$$\|\frac{f(2^mx)}{16^m}-f(x)\|\leq\frac{1}{16}\sum^{n-2}_{i=-1}\frac{\phi(0,2^{i+1}x)}{16^{i+1}},$$
from which one can prove the result by a similar technique.
\end{proof}

\begin{thm}\label{t'2} Let $s\in\{1,-1\}$ be fixed.  Suppose that an odd mapping $f:X\rightarrow Y$ satisfies

 $$\| D_f(x,y)\|\leq\phi(x,y) \eqno\hspace {0.5cm} (3.8)$$
for all $x,y\in X. $ If the upper bound $\phi:X\times
X\rightarrow[0,\infty)$ is a mapping such that
$\sum^{\infty}_{i=0} 2^{3si} \phi(0,\frac{x}{2^{si}}) $ converges,
and that $\lim_{n\rightarrow\infty}2^{3si}
\phi(\frac{x}{2^{si}},\frac{y}{2^{si}})=0$ for all $x,y\in X ,$
then the limit $C(x)=\lim_{n\rightarrow\infty} 2^{3sn}
f(\frac{x}{2^{sn}})$ exists for all $x\in X, $ and $C:X\rightarrow
Y$ is a unique cubic function satisfying (1.5),  and

$$\|f(x)-C(x)\|\leq \frac{1}{8}\sum^{\infty}_{i=\frac{s-1}{2}} 2^{3s(i+1)} \phi(0,\frac{x}{2^{s(i+1)}}) \eqno\hspace {0.5cm}(3.9)$$
for all $x\in X$.
\end{thm}
\begin{proof} Let  $s=1.$ Set $x=0$  in (3.8).  We  obtain

$$\|8f(y)-f(2y)\|\leq \phi(0,y). \eqno\hspace {0.5cm}(3.10)$$
Replacing $y$ by  $\frac{x}{2}$  in (3.10) to get

$$\parallel 8f(\frac{x}{2})-f(x)\parallel \leq
\phi(0,\frac{x}{2}). \eqno\hspace {0.5cm}(3.11)$$ An induction
argument now implies

$$\parallel 8^n f(\frac{x}{2^n})-f(x)\parallel\leq\frac{1}{8}
\sum^{n-1}_{i=0} 8^{i+1}\phi(0,\frac{x}{2^{i+1}}). \eqno\hspace
{0.5cm}(3.12)$$ Multiply (3.12) by $8^m$ and replace $x$ by
$\frac{x}{2^m},$ we obtain that
$$\parallel 8^{m+n}
f(\frac{x}{2^{m+n}})-8^mf(\frac{x}{2^m})\parallel\leq\sum^{n-1}_{i=0}
8^{m+i} \phi(0,\frac{x}{2^{m+i+1}}). \eqno\hspace {0.5cm} (3.13)$$
The right hand side of the inequality (3.13) tends to $0$ as
$m\rightarrow\infty$ because of $$\sum^{\infty}_{i=0} 8^i
\phi(0,\frac{x}{2^{i+1}})<\infty$$ by assumption,  and thus the
sequence $\{2^{3n} f(\frac{x}{2^n})\}$ is Cauchy in Y,  as
desired. Therefore we may define a mapping $C:X\rightarrow Y$ as
$C(x)=\lim_{n\rightarrow\infty} 2^{3n}{f(\frac{x}{2^n}}). $ The
rest of proof is similar to the proof of Theorem 3.1.
\end{proof}

\begin{thm}\label{t2} Let $s\in\{1,-1\}$ be fixed. Suppose a mapping  $f:X\rightarrow Y$
 satisfies $f(0)=0, $ and $\|D_f (x,y)\|\leq \phi(x,y)$ for  all $x,y \in X
.$ If the upper bound $\phi:X\times X\rightarrow [0,\infty)$ is a
mapping such that $$\sum^{\infty}_{i=0} [(|s|+s) 2^{4si}
\phi(0,\frac{x}{2^{si-1}})+(|s|-s) 2^{3si}
\phi(0,\frac{x}{2^{si-1}}) ]<\infty, \eqno\hspace {0.5cm} (3.14)$$
and

$$\lim_{n\rightarrow\infty}[(|s|+s) 2^{(4sn-1)} \phi(\frac{x}{2^{sn}},\frac{y}{2^{sn}})+(|s|-s) 2^{3sn} \phi(\frac{x}{2^{sn}},\frac{y}{2^{sn}})]=0, \eqno\hspace {0.5cm}
(3.15)$$ for all $x,y\in X. $ Then there exists a unique quartic
function $Q:X\rightarrow Y$ and a unique cubic function
$C:X\rightarrow Y$ satisfying

$$\parallel
f(x)-Q(x)-C(x)\parallel\leq\sum^{\infty}_{i=\frac{s-1}{2}}
\{(\frac{2^{4s(i+1)}}{32}+\frac{2^{3s(i+1)}}{16})
[\phi(0,\frac{x}{2^{s(i+1)}})+\phi(0,\frac{-x}{2^{s(i+1)}})]\}$$
for all $x\in X. $

\end{thm}

\begin{proof}  Let $f_e(x)=\frac{1}{2}(f(x)+f(-x))$ for all $x\in X. $  Then $f_e(0)=0$ and $f_e$ is even function
satisfying \hspace {0.3cm}
$\|D_{f_e}(x,y)\|\leq\frac{1}{2}[\phi(x,y)+\phi(-x,-y)]$ \hspace
{0.3cm} for all $x,y\in X. $ From Theorem 3.1, it follows that
there exists a unique quartic function $Q:X\rightarrow Y$
satisfies

 $$\parallel
f_e(x)-Q(x)\parallel\leq\frac{1}{32}\sum^{\infty}_{i=\frac{s-1}{2}}
\{2^{4s(i+1)} \phi(0,\frac{x}{2^{s(i+1)}})+2^{4s(i+1)}
\phi(0,\frac{-x}{2^{s(i+1)}})\} \eqno\hspace {0.5cm} (3.16)$$ for
all $x\in X. $ Let now $f_o(x)=\frac{1}{2} (f(x)-f(-x))$ for all
$x\in X. $ Then $f_o$ is odd function satisfying

$$\|D_{f_o}(x,y)\|\leq
\frac{1}{2} [\phi(x,y)+\phi(-x,-y)]$$ for all $x,y\in X. $ Hence in
view of Theorem 3.2,  it follows that there exists a unique cubic
function $C:X\rightarrow Y$ such that

 $$\parallel
f_o(x)-C(x)\parallel\leq\frac{1}{16}\sum^{\infty}_{i=\frac{s-1}{2}}
\{2^{3s(i+1)}\phi(0,\frac{x}{2^{s(i+1)}})+2^{3s(i+1)}
\phi(0,\frac{-x}{2^{s(i+1)}})\}  \eqno\hspace {0.5cm} (3.17)$$ for
all $x\in X. $ On the other hand we have $f(x)=f_e(x)+f_o(x)$ for
all $x\in X. $ Then by combining (3.16) and (3.17),  it follows
that

\begin{align*}
\parallel f(x)-C(x)-Q(x)\parallel
&\leq \hspace {0.3cm} \parallel f_e(x)-Q(x)\parallel+\parallel f_o(x)-C(x)\parallel\\
&\leq\sum^{\infty}_{i=\frac{s-1}{2}}
\{(\frac{2^{4s(i+1)}}{32}+\frac{2^{3s(i+1)}}{16})
[\phi(0,\frac{x}{2^{s(i+1)}})+\phi(0,\frac{-x}{2^{s(i+1)}})]\}
\end{align*}
for all $x\in X, $ and the proof of Theorem is complete.
\end{proof}
We are going to investigate the Hyers-Ulam -Rassias stability
problem for functional equation  (1.5).

\begin{cor}\label{t2}
 Let $p\in (-\infty,3)\bigcup(4,+\infty), $ $\theta>0. $ Suppose $f:X\rightarrow Y$ satisfies $f(0)=0
, $ and inequality
$$\|D_f (x,y)\|\leq\theta(\|x\|^p+\|y\|^p),$$
for all $x,y\in X$. Then there exists a unique quartic function
$Q:X\rightarrow Y$, and a unique cubic function $C:X\rightarrow Y$
satisfying

$$\parallel f(x)-Q(x)-C(x)\parallel \leq \begin{cases}
\theta \|x\|^p(\frac{1}{2^p-2^4}+\frac{1}{2^p-2^3}), & p>4,\\
 \theta \|x\|^p(\frac{1}{2^4-2^p}+\frac{1}{2^3-2^p}), & p<3
\end{cases} $$
for all $x\in X. $
\end{cor}
\begin{proof} Let $s=1$ in Theorem 3.3.  Then by taking
$\phi(x,y)=\theta(\|x\|^p+\|y\|^p)$ for all $x,y\in X, $ the
relations (3.14) and (3.15) hold for $p>4. $ Then there exists a
unique quartic function $Q:X\rightarrow Y$ and a unique cubic
function $C:X\rightarrow Y$ satisfying

 $$\parallel
f(x)-Q(x)-C(x)\parallel\leq\theta
\|\frac{x}{2}\|^p(\frac{1}{1-2^{4-p}}+\frac{1}{1-2^{3-p}})$$ for all
$x\in X. $ Let now $s=-1$ in Theorem 3.3 and put
$\phi(x,y)=\theta(\|x\|^p+\|y\|^p)$ for all $x,y\in X. $ Then the
relations (3.14) and (3.15) hold for $p<3. $ Then there exists a
unique quartic function $Q:X\rightarrow Y$ and a unique cubic
function $C:X\rightarrow Y$ satisfying

$$\parallel
f(x)-Q(x)-C(x)\parallel\leq
\theta\|x\|^p(\frac{1}{2^4-2^p}+\frac{1}{2^3-2^p})$$ for all $x\in
X. $
\end{proof}

Similarly, we can prove the following Ulam stability problem for
functional equation (1.5) controlled by the mixed type product-sum
function
$$(x,y)\mapsto \theta(\|x\|_X^u\|y\|_X^v+\|x\|^p+\|y\|^p)$$
introduced by J. M. Rassias (see for example [32]).

\begin{cor}\label{t2}
 Let $u,v,p$ be real numbers such that  $u+v, p\in (-\infty,3)\bigcup(4,+\infty), $ and let  $\theta>0. $ Suppose $f:X\rightarrow Y$ satisfies $f(0)=0
, $ and inequality
$$\|D_f (x,y)\|\leq\theta(\|x\|_X^u\|y\|_X^v+\|x\|^p+\|y\|^p),$$
for all $x,y\in X$. Then there exists a unique quartic function
$Q:X\rightarrow Y$, and a unique cubic function $C:X\rightarrow Y$
satisfying

$$\parallel f(x)-Q(x)-C(x)\parallel \leq \begin{cases}
\theta \|x\|^p(\frac{1}{2^p-2^4}+\frac{1}{2^p-2^3}), & p>4,\\
 \theta \|x\|^p(\frac{1}{2^4-2^p}+\frac{1}{2^3-2^p}), & p<3
\end{cases} $$
for all $x\in X. $
\end{cor}

By Corollary 3.4, we solve the following Hyers-Ulam stability
problem for functional equation  (1.5).

\begin{cor}\label{t2} Let $\epsilon$ be a positive real number. Suppose   $f:X\rightarrow Y$
satisfies $f(0)=0$, and $\|D_f(x,y)\|\leq\epsilon$, for all $x,y\in
X$. Then there exists a unique quartic function $Q:X\rightarrow Y$,
and a unique cubic function $C:X\rightarrow Y$ satisfying

$$\parallel f(x)-Q(x)-C(x)\parallel \leq\frac {22}{105} ~\epsilon$$
for all $x\in X$.
\end{cor}

\paragraph{\bf Acknowledgement.}
The authors would like to express their sincere thanks to referee
 for his invaluable comments.  The first  author would like to thank the
Semnan University for its financial support.  Also, The third author
would like to thank the office of gifted students at Semnan
University for its financial support.

{\small


}
\end{document}